\documentclass[a4paper,reqno]{amsart}

\usepackage[T1]{fontenc}
\usepackage{lmodern}
\usepackage{microtype}
\usepackage{amssymb,amsmath,amsthm,mathtools}
\usepackage{enumitem}
\usepackage{xcolor}
\usepackage[colorlinks,citecolor=blue,linkcolor=blue,urlcolor=blue]{hyperref}
\usepackage{aliascnt}
\usepackage[nameinlink,capitalise,noabbrev]{cleveref}

\allowdisplaybreaks
\numberwithin{equation}{section}
\newtheorem{theorem}{Theorem}[section]
\newaliascnt{proposition}{theorem}
\newtheorem{proposition}[proposition]{Proposition}
\aliascntresetthe{proposition}
\newaliascnt{lemma}{theorem}
\newtheorem{lemma}[lemma]{Lemma}
\aliascntresetthe{lemma}
\newaliascnt{corollary}{theorem}
\newtheorem{corollary}[corollary]{Corollary}
\aliascntresetthe{corollary}
\newaliascnt{remark}{theorem}
\newtheorem{remark}[remark]{Remark}
\aliascntresetthe{remark}
\newaliascnt{example}{theorem}
\newtheorem{example}[example]{Example}
\aliascntresetthe{example}

\theoremstyle{definition}
\newaliascnt{definition}{theorem}
\newtheorem{definition}[definition]{Definition}
\aliascntresetthe{definition}
\newaliascnt{question}{theorem}
\newtheorem{question}[question]{Question}
\aliascntresetthe{question}

\DeclareMathOperator{\add}{add}
\DeclareMathOperator{\Fac}{Fac}
\DeclareMathOperator{\Sub}{Sub}

\DeclareMathOperator{\End}{End}
\DeclareMathOperator{\Hom}{Hom}
\DeclareMathOperator{\Ext}{Ext}

\DeclareMathOperator{\rad}{rad}
\DeclareMathOperator{\topm}{top}
\DeclareMathOperator{\proj}{proj}
\DeclareMathOperator{\ind}{ind}

\DeclareMathOperator{\findim}{findim}
\DeclareMathOperator{\dell}{dell}
\DeclareMathOperator{\tr}{tr}
\DeclareMathOperator{\tors}{tors}
\DeclareMathOperator{\wide}{wide}

\newcommand{\modA}{\operatorname{mod} A}
\newcommand{\modB}{\operatorname{mod} B}
\newcommand{\stau}{\mathrm{s}\tau\text{-}\mathrm{tilt}}

\newcommand{\HH}{\mathcal H}
\newcommand{\WW}{\mathcal W}
\newcommand{\CC}{\mathcal C}
\newcommand{\EE}{\mathcal E}
\newcommand{\YY}{\mathcal Y}

\newcommand{\kk}{\Bbbk}
\newcommand{\op}{\mathrm{op}}

\title[Relative delooping levels on torsion intervals]
{Relative Delooping Levels on Intervals of Torsion Classes}

\author[W. Cao]{Weiqing Cao}
\address{School of Mathematics and Statistics, Jiangsu Normal University,
Xuzhou 221116, Jiangsu, P. R. China}
\email{caoweiqing@jsnu.edu.cn}
\author[M. Xu]{Mingfei Xu}
\address{School of Mathematics and Statistics, Jiangsu Normal University,
Xuzhou 221116, Jiangsu, P. R. China}
\email{mfxu123@163.com}
\author[X. Zhang]{Xiaojin Zhang$^*$}
\address{School of Mathematics and Statistics, Jiangsu Normal University,
Xuzhou 221116, Jiangsu, P. R. China}
\email{xjzhang@jsnu.edu.cn}
\thanks{$^*$Corresponding author.}
\date{}
\makeatletter
\@namedef{subjclassname@2020}{\textup{2020} Mathematics Subject Classification}
\makeatother

\hypersetup{
 pdftitle={Relative Delooping Levels on Intervals of Torsion Classes},
 pdfauthor={Weiqing Cao, Mingfei Xu and Xiaojin Zhang},
 pdfsubject={Relative delooping levels, torsion intervals, and finitistic dimension},
 pdfkeywords={tau-tilting module, torsion interval, exact heart, delooping level, finitistic dimension, syzygy finiteness}
}

\begin{document}

\begin{abstract}
Let $U$ be a support $\tau$-tilting module over a finite-dimensional
algebra $A$, and let $\mathcal V\subseteq\Fac U$ be any torsion class.
The exact heart $\HH=\Fac U\cap\mathcal V^\perp$ has an
Ext-progenerator $P$.  We prove that $\Hom_A(P,-)$ realizes $\HH$
as a submodule-closed exact subcategory of $\operatorname{mod}B$,
where $B=\End_A(P)$, and preserves minimal syzygies. The global
relative $k$-delooping levels are independent of $k\geq1$.
Writing $g=\operatorname{glo}\text{-}\dell\HH$, we obtain the sharp
comparison
\[
 g\leq\operatorname{glo}\text{-}\dell B\leq g+1,
 \qquad \findim B^{\op}\leq g+1.
\]
Thus the two global levels are finite simultaneously. For wide hearts the relative levels equal their classical
counterparts.  We also characterize syzygy finiteness of $B$ by a
finite relative syzygy stratum and obtain explicit bounds for both
finitistic dimensions. If the heart is Frobenius, it is exactly the
category of finitely generated Gorenstein-projective $B$-modules under
this realization, and both self-injective dimensions of $B$ are at most
one. We also give examples to show the main results.
\end{abstract}

\maketitle

\section{Introduction}

Let $A$ be a finite-dimensional algebra over a field $\kk$.  The
finitistic dimension conjecture asks whether the projective dimensions
of the finitely generated modules of finite projective dimension have a
common finite bound. G\'elinas \cite{Gelinas} introduced the delooping
level and proved
$\findim A^{\op}\leq\dell A$, where $\dell A$ is computed on simple
right $A$-modules. This relates finitistic dimension to the occurrence
of modules as stable direct summands of higher syzygies. Kershaw and Rickard \cite{KershawRickard} showed that the delooping
level can be infinite.  Guo and Igusa \cite{GuoIgusa} introduced the
derived delooping level, while Barrios, Lanzilotta and Mata
\cite{BarriosLanzilottaMata} studied global delooping levels and their
relation to other homological dimensions. Further developments concern
left--right symmetry \cite{GuoSymmetry} and derived equivalences:
Chen \cite{Chen} showed that finiteness of the ordinary delooping
level is not preserved by derived equivalences, whereas Wei, Wu and
Cao \cite{WeiWuCao} study comparisons for derived delooping with a
shift in the delooping degree.

Xu and Zhang \cite{XuZhang} defined relative delooping levels on
$\Fac U$ for a support $\tau$-tilting module $U$. Their construction
uses the Ext-progenerator $U$ and the tilting equivalence between
$\Fac U$ and a torsion-free class over $\End_A(U)$. To bound the
endomorphism algebra's delooping level, one must account for simple
modules outside the image of this equivalence. Their finite-test
invariant does so by using the radicals of the corresponding
indecomposable projective modules. Their subsequent work
\cite{XuZhangHigher} proves a comparison with one extra loop and
constructs arbitrarily large gaps between ordinary and the relative levels they defined.
Recently, Gao, Liu and Zhang \cite{GaoLiuZhang} develop derived delooping in
exact categories and apply it to $\Fac U$.  Thus the exact-category
formalism and the case of a zero lower endpoint have an established
relative theory.

In this paper we extend this construction to hearts of torsion
intervals with a functorially finite upper endpoint. The lower torsion
class may be arbitrary. To introduce the notation, first let $T$ and
$U$ be basic support
$\tau$-tilting modules with $\Fac T\subseteq\Fac U$, and put
\[
 \HH_{T,U}=\Fac U\cap T^\perp,
 \qquad T^\perp=\{X\mid\Hom_A(T,X)=0\}.
\]
Following \cite{AsaiPfeifer,EnomotoSakai}, we call $\HH_{T,U}$ the
heart of the interval. It is extension-closed, and hence carries the
exact structure inherited from $\operatorname{mod}A$. It need not be
closed under cokernels. Enomoto and Sakai
\cite[Corollary~4.19]{EnomotoSakai} proved that $U/\tr_T(U)$ is an
Ext-progenerator of this heart. Let $P$ be a basic representative of
its additive closure, set $B=\End_A(P)$, and write
$F=\Hom_A(P,-)$.  More generally, if $\mathcal V\subseteq\Fac U$
is any torsion class with torsion radical $t_{\mathcal V}$, then
$U/t_{\mathcal V}(U)$ is an Ext-progenerator of
$\Fac U\cap\mathcal V^\perp$ by
\cite[Lemma~4.17]{EnomotoSakai}. Theorem~\ref{thm:arbitrary-lower}
proves the same realization for this larger class of intervals.

The main point is to realize every such heart as a submodule-closed
subcategory of $\operatorname{mod}B$. This extends the comparison
mechanism from $\Fac U$ to intervals with a functorially finite upper
endpoint, without requiring functorial finiteness of the lower endpoint. In particular, it applies when the heart is
not ICE-closed, that is, not closed under images, cokernels, and
extensions. The restricted Yoneda construction is standard; the
submodule closure is the interval-specific ingredient.

\begin{theorem}\label{thm:intro-I}
The functor
\[
 F:\HH_{T,U}\longrightarrow\operatorname{mod}B
\]
is fully faithful and exact, its image is closed under submodules,
and it preserves minimal projective resolutions.  For each
$X\in\HH_{T,U}$ and $k\geq1$,
\[
 k\text{-}\dell_BF(X)
 \leq k\text{-}\dell_{[T,U]}X
 \leq(k+1)\text{-}\dell_BF(X).
\]
\end{theorem}

(See Proposition~\ref{prop:Yoneda},
Theorems~\ref{thm:submodule-image} and~\ref{thm:arbitrary-lower},
and Proposition~\ref{prop:sandwich}.)

Here the relative level is defined using the projectives and syzygies
of the exact heart; see Definition~\ref{def:object-dell}. The key step
in the proof is a quotient of endomorphism algebras. Set
$R=U/\tr_T(U)$ and $C=\End_A(U)$. Torsion reduction gives a
surjection $C\twoheadrightarrow\End_A(R)$. For
$G=\Hom_A(U,-)$, one has
\[
 G(\HH_{T,U})=G(\Fac U)\cap G(T)^\perp
 \quad\text{inside }\operatorname{mod}C.
\]
Both subcategories on the right are closed under submodules.
Restriction along the surjection, followed by Morita equivalence,
gives the assertion for $B$. Every first $B$-syzygy consequently
belongs to the image of $F$. This yields the comparison with one extra
loop, extending \cite[Proposition~3.3]{XuZhangHigher} from $\Fac U$
to arbitrary interval hearts.

For an indecomposable projective $B$-module $Q_i$ with top $L_i$, choose
$S_i\in\HH_{T,U}$ with $F(S_i)\cong L_i$ when $L_i$ is in the image.
For each remaining index $j$, choose $N_j\in\HH_{T,U}$ with
$F(N_j)\cong\rad Q_j$; Theorem~\ref{thm:intro-I} guarantees its existence. We define
$k\text{-}\dell^{\mathrm{XZ}}[T,U]$ to be the maximum of
$k\text{-}\dell_{[T,U]}S_i$ and
$(k+1)\text{-}\dell_{[T,U]}N_j+1$ over these two index sets.
The superscript records the correction formula of Xu--Zhang. The
global relative level, denoted
$\operatorname{glo}\text{-}k\text{-}\dell[T,U]$,
is instead the supremum over all objects of the heart.

\begin{theorem}\label{thm:intro-II}
All global relative $k$-delooping levels, $k\geq1$, coincide.
Set $g=\operatorname{glo}\text{-}\dell[T,U]$.  Then
\[
 g\leq\operatorname{glo}\text{-}\dell B\leq g+1,
\]
and, for every $k\geq1$,
\[
 k\text{-}\dell B
 \leq k\text{-}\dell^{\mathrm{XZ}}[T,U]\leq g+1.
\]
In particular, $\findim B^{\op}\leq g+1$.
If the heart is wide, then $F$ is an exact equivalence with
$\operatorname{mod}B$. The relative objectwise, finite-test, and global
levels then equal their respective classical counterparts over $B$.
\end{theorem}

(See Theorem~\ref{thm:global-sharp},
Corollary~\ref{cor:global-comparison}, and
Theorem~\ref{thm:wide-equality}.)

The independence of the global level from $k$ follows from a general
stabilization observation for exact categories with enough projectives;
see Lemma~\ref{lem:global-stabilization}. Combined with the submodule
closure of the image, it gives both inequalities in the global
comparison. In particular, the global relative level is finite if and
only if the global level of $B$ is finite. Objectwise levels can still
depend on $k$, and the finite-test and global invariants need not coincide. The correction term in Theorem~\ref{thm:intro-II} is necessary:
Example~\ref{ex:nonice} gives a connected hereditary algebra with a
nonzero lower endpoint for which every global relative level vanishes
but $\findim B^{\op}=1$.

Our next result transfers syzygy finiteness between the heart and its
endomorphism algebra.  For $r\geq0$, put
\[
 \CC_r=\add\{\Omega_{T,U}^rX\oplus P'
                 \mid X\in\HH_{T,U},\ P'\in\add P\}.
\]
Thus $\CC_0=\HH_{T,U}$.  An algebra is syzygy-finite if, for some
$r$, the additive closure of its $r$th syzygies and projectives has
only finitely many indecomposable isomorphism classes.

\begin{theorem}\label{thm:intro-III}
The algebra $B$ is syzygy-finite if and only if some $\CC_s$ is
representation-finite. If $\CC_s$ is representation-finite and contains $m_s$ indecomposable
isomorphism classes outside $\add P$, then, for every $k\geq1$,
\[
 \operatorname{glo}\text{-}k\text{-}\dell[T,U]\leq s+m_s,
 \qquad
 \max\{\findim B,\findim B^{\op}\}\leq s+m_s+1.
\]
If the heart is representation-finite, then $B$ is
torsionless-finite and $\operatorname{repdim}B\leq3$.
\end{theorem}

(See Theorem~\ref{thm:finiteness},
Proposition~\ref{prop:syzygy-transfer}, and Corollary~\ref{cor:two-sided}.)

The stabilization argument extends that of Xu--Zhang
\cite[Lemmas~5.1--5.2]{XuZhang} to arbitrary interval hearts and to any
finite relative syzygy stratum. The representation-dimension conclusion
uses Ringel's theorem for torsionless-finite algebras
\cite{RingelRepdim}. One application requires only that
$A/\operatorname{ann}_A P$ be representation-finite. In particular,
it applies to minimal representation-infinite algebras when $P$ is
not faithful. A separate local condition,
$\Ext_A^2(X,Y)=0$ for all $X,Y\in\HH_{T,U}$, implies
$\operatorname{gl.dim}B\leq2$; we give a sharp example over a hereditary
algebra with a nonzero lower endpoint.

The finite-stratum hypothesis is not needed when the heart is
Frobenius. In that case the realization identifies the heart with
$\operatorname{Gproj}B=\Sub B$, and
\[
 \operatorname{id}(B_B)\leq1,\qquad
 \operatorname{id}({}_BB)\leq1.
\]
Thus $B$ is Iwanaga--Gorenstein of dimension at most one, and all global
relative levels vanish.  The embedding of a Frobenius category into
Gorenstein-projective modules is standard \cite{ChenFrobenius}; here
submodule closure gives essential surjectivity and the dimension bound.
We realize this situation by the
monomorphism category over a self-injective algebra, using its
Frobenius structure from \cite{ChenMono}. The resulting example is
connected, representation-infinite, and of infinite global dimension;
none of its relative syzygy strata is representation-finite.  The
classical properties of monomorphism categories and triangular
algebras are used here to exhibit the scope and sharpness of the
interval comparison, rather than asserted as new results.
A Kronecker example in Section~\ref{sec:examples} also shows that
allowing a lower torsion class which is not functorially finite gives
hearts outside the class obtained from two support $\tau$-tilting
endpoints.

The paper is organized as follows.  Section~\ref{sec:preliminaries}
recalls interval hearts and proves their submodule-closed realization.
Section~\ref{sec:comparison} defines the relative invariants and proves
the comparison theorems.  Section~\ref{sec:finiteness} treats wide
intervals, finite syzygy strata, and the first stratum for ICE hearts.
Section~\ref{sec:applications} gives the applications, and
Section~\ref{sec:examples} contains further computations and questions.

Throughout the paper, $A$ is a finite-dimensional basic algebra over a field
$\kk$, and all modules are finitely generated right modules.
Subcategories are full and closed under finite direct sums,
isomorphisms, and direct summands.  Multiplication in an endomorphism
algebra is composition, so $\Hom_A(P,X)$ is a right $\End_A(P)$-module
by precomposition.  We write $\findim\Lambda$ for the supremum of the
finite projective dimensions of finitely generated right
$\Lambda$-modules.  Empty suprema of nonnegative invariants are zero;
this also covers the zero heart and the zero endomorphism algebra.

\section{Preliminaries and interval hearts}\label{sec:preliminaries}

In this section, we recall some preliminaries on torsion intervals, interval syzygies and submodule-closed realizations of intervals.

\subsection{Torsion intervals}

We write $\stau A$ for the set of isomorphism classes of basic support
$\tau$-tilting $A$-modules, and $\Fac M$ for the factors of finite
direct sums of copies of $M$. A torsion class in $\modA$ is a
subcategory closed under extensions and factor modules. We write
$\tors A$ for the poset of torsion classes ordered by inclusion. For a
subcategory $\CC$ of $\modA$, set
\[
 \CC^{\perp}
 =\{X\in\modA\mid \Hom_A(C,X)=0\text{ for all }C\in\CC\}.
\]
If $M$ is an $A$-module, then $M^{\perp}=(\add M)^{\perp}$.
A subcategory is \emph{wide} if it is closed under kernels, cokernels,
and extensions in $\modA$.

Let $T,U\in\stau A$ with $\Fac T\subseteq\Fac U$. We abbreviate this
relation to $T\leq U$.  The heart of the interval
$[\Fac T,\Fac U]$ is
\begin{equation}\label{eq:heart}
  \HH_{T,U}=\Fac U\cap(\Fac T)^{\perp}
            =\Fac U\cap T^{\perp}.
\end{equation}
The second equality follows because every object of $\Fac T$ is a factor
of an object in $\add T$. The category $\HH_{T,U}$ is closed under
extensions and images; see \cite[Remark~2.5 and Lemma~2.6]{EnomotoSakai}. Consequently
it inherits an exact structure from $\modA$: the conflations are precisely
the short exact sequences in $\modA$ all of whose terms lie in
$\HH_{T,U}$.

For modules $M,X$, recall that the trace of $M$ in $X$ is
\[
  \tr_M(X)=\sum_{f\in\Hom_A(M,X)}\operatorname{Im}f.
\]
The torsion radical of $\Fac T$ evaluated at $X$ is $\tr_T(X)$.

\begin{theorem}[Enomoto--Sakai]\label{thm:ES-progenerator}
Let $T,U\in\stau A$ with $T\leq U$. Then
\[
  R_{T,U}:=U/\tr_T(U)
\]
is an Ext-progenerator of $\HH_{T,U}$. In particular,
\[
 \proj\HH_{T,U}=\add R_{T,U},
\]
and for every $X\in\HH_{T,U}$ there is a short exact sequence
\[
 0\longrightarrow X_1\longrightarrow P_0\longrightarrow X\longrightarrow0
\]
with $P_0\in\add R_{T,U}$ and $X_1\in\HH_{T,U}$.
\end{theorem}

\begin{proof}
This is \cite[Corollary~4.19]{EnomotoSakai}, after interchanging the
letters used for the lower and upper endpoints there.
\end{proof}

\begin{remark}[Basic convention]\label{rem:basic-convention}
Fix a basic module $P_{T,U}$ with
$\add P_{T,U}=\add R_{T,U}$. Thus repeated isomorphic summands of
$U/\tr_T(U)$ are deleted. The stable category and relative syzygies
are determined by this common additive subcategory. We always form
$B_{T,U}=\End_A(P_{T,U})$ using the fixed basic representative.
\end{remark}

\subsection{Stable categories and interval syzygies}

Let $\EE$ be a full, extension-closed, summand-closed subcategory of
$\modA$ with Ext-progenerator $P$, equipped with its inherited exact
structure. In particular, $\EE$ is a Hom-finite Krull--Schmidt exact
category with enough projectives. Its projectively stable
category is
\[
  \underline{\EE}=\EE/[\add P],
\]
where $[\add P]$ is the ideal of morphisms factoring through an object of
$\add P$. We say that $X$ is a \emph{stable retract} of $Y$ if there are
maps $a:X\to Y$ and $b:Y\to X$ such that $ba-1_X$ factors through
$\add P$. Equivalently, $X$ is a direct summand of $Y\oplus P'$
for some $P'\in\add P$. Indeed, if $ba-1_X=vu$ through $P'$,
then $(a,u):X\to Y\oplus P'$ has retraction $(b,-v)$; the converse
follows on passing to the stable category.

Minimal projective deflations exist: start with any projective
deflation and remove redundant summands of its domain. The removed
summands split off its kernel, which remains in $\EE$.
For each $X\in\EE$, choose a minimal projective deflation
$p_X:P_X\twoheadrightarrow X$ and define
\[
  \Omega_{\EE}X=\ker p_X.
\]
The usual Schanuel argument in exact categories shows that this construction
induces an additive endofunctor
\[
  \Omega_{\EE}:\underline{\EE}\longrightarrow\underline{\EE};
\]
see, for example, \cite{Buhler}. We put $\Omega_{\EE}^0X=X$.

Applying this construction to $\EE=\HH_{T,U}$ and
$P=P_{T,U}$ gives the interval syzygy functor.  We write
\[
  \Omega_{T,U}=\Omega_{\HH_{T,U}}.
\]

\begin{remark}\label{rem:intrinsic-heart}
The basic Ext-progenerator of an exact Krull--Schmidt category, when it
exists, is unique up to isomorphism. Hence $\add P_{T,U}$ and
$\Omega_{T,U}$ depend only on the exact category $\HH_{T,U}$. In
particular, if two intervals have the same heart, all objectwise and global
invariants introduced below agree. Thus our basic invariant is an invariant of the exact heart.
Section~\ref{sec:examples} discusses a possible refinement using subintervals.
\end{remark}

The following restricted Yoneda lemma is the mechanism behind all later
comparisons. We record the proof because it does not require the heart to
be ICE-closed.

\begin{proposition}[Restricted Yoneda embedding]\label{prop:Yoneda}
Let $\EE$ be a full extension-closed subcategory of $\modA$ with
Ext-progenerator $P$, and put $B=\End_A(P)$. Then
\[
 F=\Hom_A(P,-):\EE\longrightarrow\modB
\]
has the following properties.
\begin{enumerate}[label=\textnormal{(\arabic*)},leftmargin=2.2em]
\item $F$ is fully faithful and exact.
\item $F$ restricts to an equivalence $\add P\simeq\proj B$.
\item If projective covers are chosen minimally, then for every $r\geq0$
      and $X\in\EE$ there is an isomorphism
      \[
        F(\Omega_{\EE}^{r}X)\cong\Omega_B^{r}F(X).
      \]
      Without fixed choices, the same assertion holds canonically in
      $\underline{\modB}$.
\item $F$ induces a fully faithful functor
      $\underline{\EE}\longrightarrow\underline{\modB}$.
\item For all $X,Y\in\EE$ and $r\geq0$ there are natural isomorphisms
      \[
       \Ext^r_{\EE}(X,Y)\cong\Ext^r_B(F(X),F(Y)),
      \]
      where the left side is computed in the exact category $\EE$.
      The image is extension-closed, and $F$ is an exact equivalence
      onto its image with the inherited exact structure.
\end{enumerate}
\end{proposition}

\begin{proof}
Since $P$ is Ext-projective in $\EE$, applying $F$ to a conflation in
$\EE$ gives a short exact sequence of $B$-modules.  Thus $F$ is exact.

To prove faithfulness, let $f:X\to Y$ satisfy $F(f)=0$.  Choose a
projective deflation $p:P_0\twoheadrightarrow X$ with $P_0\in\add P$.
Every component of $p$ is a morphism from a direct summand of a finite
sum of copies of $P$ to $X$, so $F(f)=0$ implies $fp=0$.  Since $p$ is
an epimorphism in $\modA$, we obtain $f=0$.

For fullness, let $\alpha:F(X)\to F(Y)$ be a $B$-homomorphism and take a
conflation
\[
 0\longrightarrow K\xrightarrow{i}P_0\xrightarrow{p}X\longrightarrow0
\]
with $P_0\in\add P$.  By the Yoneda isomorphism on $\add P$, the map
$\alpha F(p):F(P_0)\to F(Y)$ is $F(g)$ for some $g:P_0\to Y$.  We have
$F(gi)=0$, hence $gi=0$ by faithfulness.  Therefore $g$ factors uniquely
through $p$, say $g=fp$.  Since $F(p)$ is surjective, $F(f)=\alpha$.
This proves (1).  Part (2) follows from $F(P)=B$ and full faithfulness.

Finally, a minimal projective deflation $P_X\twoheadrightarrow X$ is
sent to a right minimal epimorphism from a projective $B$-module, hence
to a projective cover of $F(X)$. Exactness identifies the kernels.
Iteration proves (3).
For (4), every projective $B$-module is $F(P')$ for some
$P'\in\add P$. Full faithfulness therefore identifies the morphisms
factoring through $\add P$ with those factoring through projective
$B$-modules. Passing to the corresponding quotient categories proves
the assertion.

For (5), take a projective resolution of $X$ in $\EE$.
Its image is a projective resolution of $F(X)$, and full faithfulness
identifies the two Hom complexes obtained using $Y$ and $F(Y)$.
Projective resolutions compute exact-category extensions
\cite{Buhler}, so their cohomology gives the asserted isomorphisms.
For $r=1$, every extension of $F(X)$ by $F(Y)$ is therefore equivalent
to the image of a conflation in $\EE$.  This proves extension closure
and reflection of the inherited exact structure.
\end{proof}

\subsection{Submodule-closed realization for arbitrary intervals}

The following result is the extra ingredient needed to pass from an
objectwise comparison to a bound on all simple modules of the
endomorphism algebra.

\begin{theorem}[Submodule-closed image]\label{thm:submodule-image}
Let $T,U\in\stau A$ with $T\leq U$, set $P=P_{T,U}$ and
$B=\End_A(P)$.  Then the essential image of
$F=\Hom_A(P,-):\HH_{T,U}\to\modB$ is closed under submodules.
In particular, $\rad_B F(P_i)$ is in the image for every
indecomposable summand $P_i$ of $P$.
\end{theorem}

\begin{proof}
We first use the actual quotient $R=U/\tr_T(U)$, without deleting
repeated summands. Put $C=\End_A(U)$, $B'=\End_A(R)$ and
$G=\Hom_A(U,-)$. The support $\tau$-tilting module $U$ is a tilting
module over $A/\operatorname{ann}_A U$: apply
\cite[Proposition~2.2(c)]{AIR} over an idempotent quotient on which
$U$ is $\tau$-tilting. The tilting theorem therefore identifies
$\Fac U$ with the submodule-closed subcategory
$\YY=\Sub_C DU$ of $\mod C$, where $D=\Hom_{\kk}(-,\kk)$; see \cite[Theorem~2.4]{XuZhang}.
In particular, $G$ is fully faithful on $\Fac U$.

Write $tU=\tr_T(U)$ and let $q:U\twoheadrightarrow R$ be the quotient
map. The submodule $tU$ is invariant under every endomorphism of $U$.
Consequently, passage to the quotient gives an algebra homomorphism
\[
 \rho:C\longrightarrow B',\qquad c\longmapsto\overline c,
 \quad \overline c q=qc.
\]
It is surjective. Indeed, $tU\in\Fac T\subseteq\Fac U$ and $U$
is Ext-projective in $\Fac U$, so $\Ext_A^1(U,tU)=0$.
For $b\in B'$, the map $bq:U\to R$ therefore lifts along $q$ to
$c:U\to U$. Thus $qc=bq$ and $\rho(c)=b$. In particular,
\begin{equation}\label{eq:endomorphism-quotient}
 B'\cong C/\{c\in C\mid\operatorname{Im}c\subseteq tU\}.
\end{equation}

Every map $U\to X$ with $X\in\HH_{T,U}$ vanishes on $tU$.
Precomposition with $q$ induces an isomorphism
\begin{equation}\label{eq:restriction}
 \Hom_A(R,X)\xrightarrow{\sim}\Hom_A(U,X)
\end{equation}
of right $C$-modules, where the left side is restricted along $\rho$.
Both $T$ and $X$ belong to $\Fac U$. Full faithfulness of $G$ gives
\[
 G(\HH_{T,U})=\YY\cap G(T)^\perp
 \quad\text{inside }\mod C.
\]
Both subcategories on the right are closed under submodules.

Now let $L$ be a $B'$-submodule of $\Hom_A(R,X)$. Restriction along
$\rho$ and \eqref{eq:restriction} identify it with a $C$-submodule of
$G(X)$, hence with $G(X')$ for some $X'\in\HH_{T,U}$.
Equation~\eqref{eq:restriction} identifies $G(X')$ with
$\Hom_A(R,X')$. Since $\rho$ is surjective, a $C$-linear map between
$B'$-modules is $B'$-linear. Thus
$L\cong\Hom_A(R,X')$ as $B'$-modules, proving the assertion for $R$.

Finally, $\add R=\add P$. The standard Morita equivalence between
$\mod B'$ and $\mod B$ identifies the two restricted Yoneda functors.
Exact equivalences preserve subobjects, so the assertion holds for the
chosen basic representative $P$. The last statement follows because
$F(P_i)$ belongs to the image.
\end{proof}

\begin{remark}\label{rem:scope}
The theorem uses no cokernel closure of the heart. Thus a separate
condition requiring radical lifts, even only for projectives with
missing simple tops, is automatic for the intervals considered here.
Example~\ref{ex:nonice} exhibits this phenomenon for a non-ICE heart.
\end{remark}

\begin{theorem}[Arbitrary lower torsion classes]\label{thm:arbitrary-lower}
Let $U\in\stau A$, let $\mathcal V\subseteq\Fac U$ be any torsion
class, and denote its torsion radical by $t_{\mathcal V}$. Put
\[
 \HH_{\mathcal V,U}=\Fac U\cap\mathcal V^\perp,
 \qquad R=U/t_{\mathcal V}(U).
\]
Then $R$ is an Ext-progenerator of $\HH_{\mathcal V,U}$.
If $P$ is a basic representative of $\add R$ and $B=\End_A(P)$, the
functor $\Hom_A(P,-)$ identifies $\HH_{\mathcal V,U}$ with a
submodule-closed, extension-closed exact subcategory of $\modB$
containing all projectives. It preserves minimal projective
resolutions. No functorial finiteness of $\mathcal V$ is required.
\end{theorem}

\begin{proof}
The Ext-progenerator assertion is
\cite[Lemma~4.17]{EnomotoSakai}. Proposition~\ref{prop:Yoneda}
therefore applies; it remains to establish submodule closure.
Set $C=\End_A(U)$ and $G=\Hom_A(U,-)$. As in the proof of
Theorem~\ref{thm:submodule-image}, $G$ identifies $\Fac U$ with
$\YY=\Sub_C DU$. Since $\mathcal V\subseteq\Fac U$, full
faithfulness gives
\[
 G(\HH_{\mathcal V,U})=\YY\cap G(\mathcal V)^\perp.
\]
This is an intersection of submodule-closed subcategories.
The torsion submodule $t_{\mathcal V}(U)$ is invariant under every
endomorphism of $U$. Moreover,
$\Ext_A^1(U,t_{\mathcal V}(U))=0$, so quotienting induces a
surjection $C\twoheadrightarrow\End_A(R)$. Every map from $U$ to
an object of $\HH_{\mathcal V,U}$ vanishes on $t_{\mathcal V}(U)$.
Thus restriction along this surjection identifies
$\Hom_A(R,-)$ with $G$ on the heart. Submodule closure follows by
restriction and then Morita equivalence, exactly as in
Theorem~\ref{thm:submodule-image}.
\end{proof}

\begin{remark}\label{rem:general-lower-scope}
The comparison results of Section~\ref{sec:comparison}, the
stabilization criterion, the syzygy-finiteness transfer, and the
Frobenius result use only the realization in
Theorem~\ref{thm:arbitrary-lower}. They consequently hold for
$\HH_{\mathcal V,U}$ with the same definitions and proofs. We retain
$[T,U]$ in their statements for readability. The canonical wide
realization used for ICE hearts in Section~\ref{sec:finiteness}
continues to be stated with two functorially finite endpoints.
\end{remark}

\section{Relative delooping levels and comparison theorems}\label{sec:comparison}

In this section, we show the main theorems on relative delooping level over intervals and give the comparison theorem.

\subsection{Objectwise and global levels}

We specialize the ordinary exact-category definition
\cite[Definition~2.4]{GaoLiuZhang} to the heart of an interval with
support $\tau$-tilting endpoints. The subsequent comparison uses the
submodule-closed realization proved in
Theorem~\ref{thm:submodule-image}.

\begin{definition}\label{def:object-dell}
Let $T,U\in\stau A$ with $T\leq U$, let $X\in\HH_{T,U}$, and let
$k\geq1$.
\begin{enumerate}[label=\textnormal{(\arabic*)},leftmargin=2.2em]
\item The \emph{$k$-delooping level of $X$ relative to the interval}
      $[T,U]$, denoted
      \[
        k\text{-}\dell_{[T,U]}X,
      \]
      is the infimum of the integers $d\geq0$ for which there is an
      object $Y\in\HH_{T,U}$ such that $\Omega_{T,U}^{d}X$ is a stable
      retract of $\Omega_{T,U}^{d+k}Y$ in
      $\underline{\HH}_{T,U}$.  The infimum of the empty set is
      $\infty$.
\item The \emph{global $k$-delooping level of the interval heart} is
      \[
      \operatorname{glo}\text{-}k\text{-}\dell[T,U]
      =\sup_{X\in\HH_{T,U}} k\text{-}\dell_{[T,U]}X.
      \]
\end{enumerate}
For $k=1$ we omit $k$ from the terminology and notation whenever no
confusion is possible.
\end{definition}

For later comparison, if $\Lambda$ is a finite-dimensional algebra, we
write
\[
 \operatorname{glo}\text{-}k\text{-}\dell\Lambda
 =\sup_{M\in\operatorname{mod}\Lambda}k\text{-}\dell_{\Lambda}M
\]
for its classical global $k$-delooping level. Here the classical objectwise
level is Definition~\ref{def:object-dell} applied to
$\operatorname{mod}\Lambda$ and its ordinary projectives, whereas
\[
 k\text{-}\dell\Lambda
 =\sup\{k\text{-}\dell_\Lambda L\mid L\text{ is simple in }
          \operatorname{mod}\Lambda\}.
\]
These two suprema are not identified in this paper.

\begin{remark}[Boundary cases]\label{rem:boundary}
If $T=0$, then $\HH_{0,U}=\Fac U$ and $P_{0,U}=U$. Thus
Definition~\ref{def:object-dell} recovers the objectwise and global
relative levels of Xu--Zhang
\cite[Definition~4.4(1),(3)]{XuZhang}. At the other extreme, if $T=U$,
then $\HH_{T,U}=0$ and $P_{T,U}=0$, so the construction is vacuous.  If
$U=A$, then
\[
 \HH_{T,A}=T^{\perp}
 \qquad\text{and}\qquad
 P_{T,A}=A/\tr_T(A).
\]
Thus a nonzero lower endpoint can change the exact category and its
relative projectives. The construction is not merely a relabeling of
the case $T=0$.
\end{remark}

\begin{lemma}[Monotonicity in the loop length]\label{lem:k-monotonicity}
For every $X\in\HH_{T,U}$ and $k\geq1$,
\[
  k\text{-}\dell_{[T,U]}X
  \leq
  (k+1)\text{-}\dell_{[T,U]}X.
\]
Consequently,
\[
 \operatorname{glo}\text{-}k\text{-}\dell[T,U]
 \leq
 \operatorname{glo}\text{-}(k+1)\text{-}\dell[T,U].
\]
\end{lemma}

\begin{proof}
If $\Omega_{T,U}^{d}X$ is a stable retract of
$\Omega_{T,U}^{d+k+1}Y$, then it is a stable retract of
\[
 \Omega_{T,U}^{d+k}(\Omega_{T,U}Y).
\]
The same $d$ is therefore admissible for the $k$-delooping level. Taking
suprema gives the second assertion.
\end{proof}

\begin{lemma}[Global levels and stabilization]\label{lem:global-stabilization}
Let $\EE$ be an exact category as in Proposition~\ref{prop:Yoneda},
with projectives $\add P$. For $r\geq0$, put
\[
 \CC_r(\EE)=\add\{\Omega_{\EE}^rX\oplus P'
       \mid X\in\EE,\ P'\in\add P\}.
\]
Define its objectwise and global levels as in
Definition~\ref{def:object-dell}. For every $k\geq1$,
\begin{equation}\label{eq:global-stabilization}
 \operatorname{glo}\text{-}k\text{-}\dell\EE
 =\operatorname{glo}\text{-}\dell\EE
 =\inf\{r\geq0\mid\CC_r(\EE)=\CC_{r+1}(\EE)\}.
\end{equation}
The infimum is $\infty$ if the strata never stabilize.
In particular, this applies both to interval hearts and to module
categories of finite-dimensional algebras.
\end{lemma}

\begin{proof}
The strata descend and satisfy
\[
 \CC_{r+1}(\EE)
 =\add\bigl(\Omega_{\EE}(\CC_r(\EE))\cup\add P\bigr).
\]
Suppose first that $\operatorname{glo}\text{-}\dell\EE\leq n$.
For every $X\in\EE$, a delooping witness in degree $d\leq n$ can be
raised to degree $n$ by applying $\Omega_{\EE}^{n-d}$. Thus
$\Omega_{\EE}^nX$ is a stable retract of some
$\Omega_{\EE}^{n+1}Y$, which implies
$\CC_n(\EE)\subseteq\CC_{n+1}(\EE)$. The reverse inclusion always
holds, so the strata are equal.

Conversely, equality at index $n$ propagates by the displayed identity
to $\CC_n(\EE)=\CC_{n+k}(\EE)$ for every $k\geq1$.
Consequently, $\Omega_{\EE}^nX$ is a direct summand of
$\Omega_{\EE}^{n+k}Y\oplus P'$ for some $Y\in\EE$ and
$P'\in\add P$. Here finite sums of possible witnesses can be
combined into one $Y$ by additivity. This proves
$\operatorname{glo}\text{-}k\text{-}\dell\EE\leq n$ for every $k$.
Finally, the objectwise monotonicity of
Lemma~\ref{lem:k-monotonicity} applies in $\EE$. Hence all the
global levels equal the first stabilization index, including when that
index is infinite.
\end{proof}

\begin{remark}
Lemma~\ref{lem:global-stabilization} concerns a supremum over all
objects. It does not identify the objectwise levels for different
$k$, nor the invariants obtained by testing only simple modules.
We write $g=\operatorname{glo}\text{-}\dell[T,U]$ for the common
global relative value whenever the interval is fixed.
\end{remark}

\begin{proposition}[Objectwise comparison]\label{prop:object-comparison}
Let $P=P_{T,U}$, $B=\End_A(P)$ and
$F=\Hom_A(P,-)$. For every $X\in\HH_{T,U}$ and every $k\geq1$,
\[
  k\text{-}\dell_B F(X)
  \leq
  k\text{-}\dell_{[T,U]}X.
\]
\end{proposition}

\begin{proof}
There is nothing to prove if the right-hand side is infinite. Otherwise,
let $d=k\text{-}\dell_{[T,U]}X$. There is $Y\in\HH_{T,U}$ such that
$\Omega_{T,U}^{d}X$ is a stable retract of
$\Omega_{T,U}^{d+k}Y$. Proposition~\ref{prop:Yoneda} sends this stable
retract to
\[
 \Omega_B^dF(X)\mid\Omega_B^{d+k}F(Y)
 \quad\text{in }\underline{\modB}.
\]
The classical definition allows the delooping object to range over all
of $\modB$, so the assertion follows.
\end{proof}

The proof of Proposition~\ref{prop:object-comparison} supplies only one
inequality: a classical delooping witness for $F(X)$ need not belong
to the essential image of $F$. More importantly, the proposition alone
does not bound $k\text{-}\dell B$, because some simple $B$-modules may
not lie in $F(\HH_{T,U})$. Theorem~\ref{thm:submodule-image} provides the radical lifts needed
to treat the missing simple modules.

\begin{proposition}[Comparison with one extra loop]\label{prop:sandwich}
For $X\in\HH_{T,U}$ and $k\geq1$,
\[
 k\text{-}\dell_BF(X)
 \leq k\text{-}\dell_{[T,U]}X
 \leq (k+1)\text{-}\dell_BF(X).
\]
\end{proposition}

\begin{proof}
Only the second inequality remains. By
Theorem~\ref{thm:submodule-image}, $\Omega_B M$ is in the image of $F$
for every $M\in\modB$, since it is a submodule of a projective.
If $d=(k+1)\text{-}\dell_BF(X)<\infty$, choose $M$ with
\[
 \Omega_B^dF(X)\mid\Omega_B^{d+k+1}M
\]
in the stable category, and choose $Y\in\HH_{T,U}$ with
$F(Y)\cong\Omega_BM$. The right-hand side is
$F(\Omega_{T,U}^{d+k}Y)$. By Proposition~\ref{prop:Yoneda}(4), $F$ is fully faithful on
projectively stable categories. We can thus pull back the stable
retract to obtain
$\Omega_{T,U}^dX\mid\Omega_{T,U}^{d+k}Y$.
This proves the assertion; the case $d=\infty$ is immediate.
\end{proof}

For $T=0$, Proposition~\ref{prop:sandwich} is
\cite[Proposition~3.3]{XuZhangHigher}. The same proof works here
because Theorem~\ref{thm:submodule-image} puts every first
$B$-syzygy in the image for arbitrary lower endpoints.

\begin{theorem}[Sharp global comparison]\label{thm:global-sharp}
Let $T,U\in\stau A$ with $T\leq U$, and put
$B=\End_A(P_{T,U})$ and $g=\operatorname{glo}\text{-}\dell[T,U]$.
Then, in the extended nonnegative integers,
\begin{equation}\label{eq:global-sharp}
 g\leq\operatorname{glo}\text{-}\dell B\leq g+1.
\end{equation}
In particular, these two global levels are finite simultaneously;
when finite, they are equal or differ by one. The additive constant
one is sharp even for connected hereditary $A$ and nonzero $T$.
The same assertions hold for the hearts of
Theorem~\ref{thm:arbitrary-lower}.
\end{theorem}

\begin{proof}
For every $X\in\HH_{T,U}$, Proposition~\ref{prop:sandwich} gives
\[
 \dell_{[T,U]}X\leq2\text{-}\dell_BF(X)
 \leq\operatorname{glo}\text{-}2\text{-}\dell B
 =\operatorname{glo}\text{-}\dell B,
\]
where the equality is Lemma~\ref{lem:global-stabilization}.
Taking the supremum over $X$ proves the first inequality.

For the second, assume $g<\infty$, take $M\in\modB$, and choose
$X\in\HH_{T,U}$ with $F(X)\cong\Omega_BM$.
Such an $X$ exists by Theorem~\ref{thm:submodule-image}.
Lemma~\ref{lem:global-stabilization} gives
$2\text{-}\dell_{[T,U]}X\leq g$. Raising a witness to degree $g$
if necessary, choose $Y$ such that
$\Omega_{T,U}^gX$ is a stable retract of
$\Omega_{T,U}^{g+2}Y$.  Applying $F$ gives
\[
 \Omega_B^{g+1}M\mid\Omega_B^{g+2}F(Y)
 \quad\text{in }\underline{\modB}.
\]
Thus $\dell_BM\leq g+1$, proving the second inequality.
Example~\ref{ex:nonice} has $g=0$ and
$\operatorname{glo}\text{-}\dell B=1$, proving sharpness.
The proof uses only the realization of the heart, so
Theorem~\ref{thm:arbitrary-lower} gives the last assertion.
\end{proof}

\subsection{A finite-test invariant}

Fix a decomposition
\[
 P=P_1\oplus\cdots\oplus P_n
\]
into pairwise nonisomorphic indecomposable summands.  Put
\[
 Q_i=F(P_i),\qquad L_i=\topm_BQ_i.
\]
The modules $Q_1,\dots,Q_n$ are the indecomposable projective right
$B$-modules, and $L_1,\dots,L_n$ form a complete list of simple right
$B$-modules.

Let
\[
 I=I(T,U)=\{i\mid L_i\in F(\HH_{T,U})\}.
\]
For $i\in I$, full faithfulness gives an object $S_i\in\HH_{T,U}$,
unique up to isomorphism, such that $F(S_i)\cong L_i$.

By Theorem~\ref{thm:submodule-image}, for each $j\notin I$ there is a
unique object $N_j\in\HH_{T,U}$, up to isomorphism, with
$F(N_j)\cong\rad_B Q_j$.

\begin{definition}\label{def:XZ-level}
For $k\geq1$, define the \emph{Xu--Zhang $k$-delooping level of the
interval} by
\begin{equation}\label{eq:XZ-level}
\begin{aligned}
 k\text{-}\dell^{\mathrm{XZ}}[T,U]
 =\max\Big\{&
 \sup_{i\in I} k\text{-}\dell_{[T,U]}S_i,\\
 &\sup_{j\notin I}
 \big((k+1)\text{-}\dell_{[T,U]}N_j+1\big)
 \Big\}.
\end{aligned}
\end{equation}
\end{definition}

For $T=0$, the objects $S_i$ and $N_j$ are the simple-image and
radical-correction objects of \cite[Definition~4.4(2)]{XuZhang}.
Thus the definition recovers their finite-test invariant, as well as
extending their objectwise and global invariants.

The shift from $k$ to $k+1$ in the second line accounts for the exact
sequence $0\to\rad Q_j\to Q_j\to L_j\to0$.
Example~\ref{ex:correction} shows that the radical correction term is
necessary. The separate issue of the loop-length shift is addressed
by the gap examples of \cite{XuZhangHigher}.

\begin{theorem}[Interval comparison theorem]\label{thm:main-comparison}
Let $T,U\in\stau A$ with $T\leq U$.  Put $P=P_{T,U}$ and $B=\End_A(P)$.  Then, for every
$k\geq1$,
\[
   k\text{-}\dell B
   \leq
   k\text{-}\dell^{\mathrm{XZ}}[T,U].
\]
In particular,
\[
   \findim B^{\op}
   \leq
   \dell^{\mathrm{XZ}}[T,U].
\]
\end{theorem}

\begin{proof}
Let $L_i$ be a simple $B$-module as above. If $i\in I$, then
$L_i\cong F(S_i)$ and Proposition~\ref{prop:object-comparison} gives
\[
  k\text{-}\dell_B L_i
  \leq k\text{-}\dell_{[T,U]}S_i.
\]

Now let $j\notin I$. Theorem~\ref{thm:submodule-image} gives $N_j\in\HH_{T,U}$ and an
exact sequence
\[
 0\longrightarrow F(N_j)\longrightarrow Q_j
   \longrightarrow L_j\longrightarrow0.
\]
Since $Q_j\twoheadrightarrow L_j$ is a projective cover,
\begin{equation}\label{eq:first-syzygy-missing}
  \Omega_BL_j\cong F(N_j).
\end{equation}
Suppose that $(k+1)\text{-}\dell_{[T,U]}N_j=d<\infty$. For some
$Y\in\HH_{T,U}$ there is a stable retract
\[
  \Omega_{T,U}^{d}N_j
  \mid
  \Omega_{T,U}^{d+k+1}Y.
\]
Applying $F$, Proposition~\ref{prop:Yoneda} and
\eqref{eq:first-syzygy-missing} yield
\[
  \Omega_B^{d+1}L_j
  \mid
  \Omega_B^{d+k+1}F(Y)
  =\Omega_B^{(d+1)+k}F(Y).
\]
Thus
\[
  k\text{-}\dell_BL_j\leq d+1.
\]
The same inequality is vacuous if $d=\infty$. Taking the supremum over
all simple $B$-modules proves the first assertion. For $k=1$, combine it
with G\'elinas' inequality
$\findim B^{\op}\leq\dell B$ \cite[Proposition~1.3]{Gelinas}.
\end{proof}

\begin{corollary}[Finite-test and finitistic bounds]\label{cor:global-comparison}
Under the hypotheses of Theorem~\ref{thm:main-comparison}, put
$g=\operatorname{glo}\text{-}\dell[T,U]$. For every $k\geq1$,
\[
 k\text{-}\dell B
 \leq k\text{-}\dell^{\mathrm{XZ}}[T,U]\leq g+1.
\]
In particular,
\[
 \findim B^{\op}\leq\dell^{\mathrm{XZ}}[T,U]\leq g+1.
\]
\end{corollary}

\begin{proof}
By Lemma~\ref{lem:global-stabilization}, all global relative levels
in \eqref{eq:XZ-level} equal $g$. Hence the first supremum in that
formula is at most $g$ and the second is at most $g+1$.
Apply Theorem~\ref{thm:main-comparison}.
\end{proof}

\begin{remark}[The big finitistic dimension]\label{rem:big-findim}
Let $\operatorname{Findim}\Lambda$ denote the supremum of the finite
projective dimensions of all right $\Lambda$-modules. G\'elinas'
Artinian bound \cite[Proposition~1.3]{Gelinas} is in fact
$\operatorname{Findim}B^{\op}\leq\dell B$. Consequently,
\[
 \operatorname{Findim}B^{\op}
 \leq\dell^{\mathrm{XZ}}[T,U]
 \leq\operatorname{glo}\text{-}\dell[T,U]+1.
\]
Thus all later upper bounds for $\findim B^{\op}$ also hold for
$\operatorname{Findim}B^{\op}$. The bound for $\findim B$ in
Corollary~\ref{cor:two-sided} uses a separate argument with finitely
generated modules.
\end{remark}

\section{Wide intervals and finiteness criteria}\label{sec:finiteness}

In this section, we study the relative delooping level over wide intervals and give applications to finitistic dimension.

\subsection{Wide intervals}

The interval $[\Fac T,\Fac U]$ is called \emph{wide} when its heart is
a wide subcategory of $\modA$, following \cite{AsaiPfeifer}. Since the endpoints are functorially finite, the heart is a
functorially finite wide subcategory; see
\cite{AsaiPfeifer,EnomotoSakai}.

\begin{theorem}[Wide interval equality]\label{thm:wide-equality}
Suppose that $[\Fac T,\Fac U]$ is a wide interval. Let
$P=P_{T,U}$, $B=\End_A(P)$ and $F=\Hom_A(P,-)$. Then $P$ is a
projective generator of the abelian category $\HH_{T,U}$, and
\[
   F:\HH_{T,U}\xrightarrow{\sim}\modB
\]
is an exact equivalence. Consequently, for every $X\in\HH_{T,U}$ and
$k\geq1$,
\[
 k\text{-}\dell_{[T,U]}X
 =k\text{-}\dell_BF(X).
\]
Moreover,
\[
 k\text{-}\dell^{\mathrm{XZ}}[T,U]=k\text{-}\dell B
\]
and
\[
 \operatorname{glo}\text{-}k\text{-}\dell[T,U]
 =\operatorname{glo}\text{-}k\text{-}\dell B,
\]
where the right-hand side of the last equality is the supremum over all
finitely generated $B$-modules.
\end{theorem}

\begin{proof}
The exact structure inherited by the wide subcategory $\HH_{T,U}$ is its
abelian exact structure. Theorem~\ref{thm:ES-progenerator} therefore
says that $P$ is a projective generator. Morita theory gives the exact
equivalence $F$.

The equivalence identifies projective covers, stable categories, and
syzygy functors. Since every $B$-module is in the image, the module used
to deloop $F(X)$ can be pulled back to $\HH_{T,U}$; this gives the
objectwise equality. Every simple $B$-module lies in the image, so
$I(T,U)=\{1,\dots,n\}$ and there are no radical correction terms in
\eqref{eq:XZ-level}. Taking suprema over the simple modules and then
over all modules proves the remaining equalities.
\end{proof}

\subsection{A finiteness criterion}

We extend the stabilization argument of
\cite[Lemmas~5.1--5.2]{XuZhang} to arbitrary hearts and to any finite
relative syzygy stratum. Put $\CC_0=\HH_{T,U}$ and, for $r\geq1$,
\[
 \CC_r=\add\{\Omega_{T,U}^r X\oplus P'
       \mid X\in\HH_{T,U},\ P'\in\add P\}.
\]
Let $\mathcal L_r(T,U)=\ind\CC_r$. These sets form a descending
sequence. Projective summands are included in every stratum.

\begin{lemma}[Uniform stabilization]\label{lem:stabilization}
Suppose that $\mathcal L_s(T,U)$ is finite for some $s\geq0$ and put
\[
 n=\min\{r\geq s\mid\mathcal L_r(T,U)=\mathcal L_{r+1}(T,U)\}.
\]
Then $n$ exists, and
\[
 n=\max\{s,\operatorname{glo}\text{-}\dell[T,U]\}.
\]
In particular, $\operatorname{glo}\text{-}k\text{-}\dell[T,U]\leq n$
for every $k\geq1$. If $m_s$ is the number of classes in $\mathcal L_s(T,U)$ represented
by objects outside $\add P$, then $n\leq s+m_s$.
\end{lemma}

\begin{proof}
The sets $\mathcal L_r(T,U)$ descend, and any equality of successive
strata propagates to all later strata, by
Lemma~\ref{lem:global-stabilization}. Since $\mathcal L_s(T,U)$ is
finite, such an equality must occur. The same lemma identifies its
first index starting from zero with
$\operatorname{glo}\text{-}\dell[T,U]$, proving the formula for $n$.
Before stabilization each strict inclusion removes at least one
nonprojective isomorphism class. Projective classes persist, so there
are at most $m_s$ strict inclusions after index $s$. Hence
$n\leq s+m_s$.
\end{proof}

\begin{theorem}[Interval finiteness criterion]\label{thm:finiteness}
Let $T,U\in\stau A$ with $T\leq U$, and set $B=\End_A(P_{T,U})$.
If $\mathcal L_s(T,U)$ is finite for some $s\geq0$, then, with $n$
as in Lemma~\ref{lem:stabilization},
\[
 k\text{-}\dell B\leq k\text{-}\dell^{\mathrm{XZ}}[T,U]
 \leq n+1\leq s+m_s+1\qquad(k\geq1).
\]
In particular, $\findim B^{\op}\leq n+1<\infty$.
If the interval is wide, the bound $n+1$ can be replaced by $n$.
\end{theorem}

\begin{proof}
Lemma~\ref{lem:stabilization} bounds every term
$k\text{-}\dell_{[T,U]}S_i$ and
$(k+1)\text{-}\dell_{[T,U]}N_j$ by $n$.
Apply \eqref{eq:XZ-level} and Theorem~\ref{thm:main-comparison}.
For a wide interval there are no radical correction terms, by
Theorem~\ref{thm:wide-equality}.
\end{proof}

\begin{corollary}[Finite representation type]
\label{cor:ambient-repfinite}
If $A$ is representation-finite, then
\[
 \findim\End_A(P_{T,U})^{\op}<\infty
\]
for every interval with support $\tau$-tilting endpoints.
More generally, the same conclusion holds whenever the heart is
representation-finite. If the heart has $m$ indecomposables that are
not relatively projective, the bound $m+1$ is valid, and $m$ is valid for a wide
heart.
\end{corollary}

\begin{proof}
Use $s=0$ in Theorem~\ref{thm:finiteness}. A summand-closed subcategory
of a representation-finite module category is representation-finite.
\end{proof}

\subsection{The first stratum for ICE hearts}

A subcategory is \emph{ICE-closed} if it is closed under images,
cokernels, and extensions. An interval is an ICE interval when its
heart is ICE-closed. In this case Enomoto--Sakai
\cite[Corollary~4.19 and Proposition~4.14]{EnomotoSakai} show that $P$
is a wide $\tau$-tilting module and that
\[
 \WW=\langle P\rangle_{\wide},\qquad
 \HH_{T,U}=\Fac_{\WW}P
\]
is its canonical functorially finite wide realization. Thus the heart
is a torsion class in $\WW$. This description is not needed for
Theorem~\ref{thm:submodule-image}, but it is useful for recognizing the
first syzygy stratum.

There is a useful categorical description of the first syzygy stratum.
If $\WW$ is a wide subcategory and $P\in\WW$, let $\Sub_{\WW}P$ denote
the objects of $\WW$ which embed, in $\WW$, into an object of $\add P$.

\begin{proposition}[First syzygy stratum]\label{prop:L1}
Let $[\Fac T,\Fac U]$ be an ICE interval, set
$P=P_{T,U}$, and let $\WW=\langle P\rangle_{\wide}$ be the canonical
wide subcategory. Then
\[
  \add\mathcal L_1(T,U)
  =\HH_{T,U}\cap\Sub_{\WW}P.
\]
\end{proposition}

\begin{proof}
Every first relative syzygy is a subobject in $\WW$ of an object of
$\add P$, so the left-hand side is contained in the right-hand side.

Conversely, let $X\in\HH_{T,U}\cap\Sub_{\WW}P$. There is a monomorphism
$X\hookrightarrow P'$ in $\WW$ with $P'\in\add P$. Since
$\HH_{T,U}=\Fac_{\WW}P$ is a torsion class in $\WW$, the cokernel $Y$
belongs to $\HH_{T,U}$. Thus
\[
 0\longrightarrow X\longrightarrow P'\longrightarrow Y\longrightarrow0
\]
is a conflation in $\HH_{T,U}$.  Comparing it with a minimal projective
presentation
\[
 0\longrightarrow\Omega_{T,U}Y\longrightarrow P_Y
 \longrightarrow Y\longrightarrow0
\]
and applying Schanuel's lemma gives
\[
 X\oplus P_Y\cong\Omega_{T,U}Y\oplus P'.
\]
Therefore $X\in\add\mathcal L_1(T,U)$.
\end{proof}

\begin{corollary}[Representation-finiteness criterion]\label{cor:repfinite}
Under the hypotheses of Proposition~\ref{prop:L1}, if
\[
  \HH_{T,U}\cap\Sub_{\WW}P
\]
is of finite representation type, then
\[
  \findim\End_A(P)^{\op}<\infty.
\]
In particular, the conclusion holds if $\HH_{T,U}$ is of finite
representation type.
\end{corollary}

\begin{proof}
Proposition~\ref{prop:L1} implies that $\mathcal L_1(T,U)$ is finite,
so Theorem~\ref{thm:finiteness} applies.
\end{proof}

\section{Applications beyond representation-finite ambient algebras}
\label{sec:applications}

In this section, we give some application of the comparison theorems. We first strengthen the
finite-stratum criterion, then give local sufficient conditions and a
Frobenius example to which no finite-stratum criterion applies.

\subsection{Syzygy finiteness and both finitistic dimensions}

For an algebra $B$, put
\[
 \mathcal D_r(B)=\add\{\Omega_B^rM\oplus Q
       \mid M\in\modB,\ Q\in\proj B\}\qquad(r\geq0).
\]
We call $B$ \emph{syzygy-finite} if some $\mathcal D_r(B)$ is
representation-finite.  The case $r=1$ is called
\emph{torsionless-finite}, since $\mathcal D_1(B)=\Sub B$.
Recall the relative strata $\CC_r$ from Section~\ref{sec:finiteness}.

\begin{proposition}[Transfer of syzygy finiteness]\label{prop:syzygy-transfer}
For every $r\geq0$ one has
\[
 \mathcal D_{r+1}(B)\subseteq F(\CC_r)\subseteq\mathcal D_r(B).
\]
Consequently, $B$ is syzygy-finite if and only if some relative stratum
$\CC_s$ is representa\-tion-finite. If the heart is representation-finite,
then $B$ is torsionless-finite and
\[
 \operatorname{repdim}B\leq3.
\]
\end{proposition}

\begin{proof}
The second inclusion is the preservation of syzygies in
Proposition~\ref{prop:Yoneda}. Every $\Omega_BM$ is in the image of
$F$, by Theorem~\ref{thm:submodule-image}. Write
$\Omega_BM\cong F(X)$ with $X\in\HH_{T,U}$. Then
\[
 \Omega_B^{r+1}M\cong F(\Omega_{T,U}^rX),
\]
which proves the first inclusion, including direct summands and
projectives. Full faithfulness identifies indecomposable isomorphism
classes in $\CC_r$ with those in $F(\CC_r)$, proving the equivalence.
For a finite heart use $r=0$. The representation-dimension bound is
Ringel's theorem for torsionless-finite algebras \cite{RingelRepdim};
here $\operatorname{repdim}B$ is the infimum of
$\operatorname{gl.dim}\End_B(V)$ over generator-cogenerators $V$.
\end{proof}

\begin{corollary}[A bound on both sides]\label{cor:two-sided}
If $\CC_s$ is representation-finite and contains $m_s$ isomorphism
classes of indecomposable objects that are not relatively projective, then
\[
 \max\{\findim B,\findim B^{\op}\}\leq s+m_s+1.
\]
\end{corollary}

\begin{proof}
The opposite bound is Theorem~\ref{thm:finiteness}.
Let $M\in\modB$ have finite projective dimension. By
Proposition~\ref{prop:syzygy-transfer}, its $(s+1)$st syzygy lies in
$F(\CC_s)$. This category is closed under syzygies and direct summands
and has $m_s$ nonprojective indecomposable isomorphism classes.
Every object of finite projective dimension in this category has
projective dimension at most $m_s$: an indecomposable of projective
dimension $d>0$ has an indecomposable summand of its syzygy of projective
dimension $d-1$ when $d>1$, and iteration gives $d$ distinct
nonprojective classes, distinguished by their projective dimensions.
Apply this to the summands of $\Omega_B^{s+1}M$. Thus
$\operatorname{pd}_BM\leq s+1+m_s$, as required.
\end{proof}

\subsection{Representation-finite annihilator quotients}

The ambient algebra need not be representation-finite. It suffices to
check a quotient determined by the interval projective generator.

\begin{corollary}[Annihilator criterion]\label{cor:annihilator}
Set $I=\operatorname{ann}_A P_{T,U}$.  If $A/I$ is
representation-finite, then $\HH_{T,U}$ is representation-finite,
$B$ is torsionless-finite, and
\[
 \operatorname{repdim}B\leq3,\qquad
 \findim B<\infty,\qquad\findim B^{\op}<\infty.
\]
More precisely, if the heart has $m$ indecomposables that are not
relatively projective,
then both finitistic dimensions are at most $m+1$.
The same conclusion holds if some ideal $J$ satisfies $PJ=0$ and
$A/J$ is representation-finite.
\end{corollary}

\begin{proof}
Every $X\in\HH_{T,U}$ is a quotient of an object of $\add P$,
so $XI=0$. Thus the heart is a full summand-closed subcategory of
$\operatorname{mod}(A/I)$ and is representation-finite. Apply
Proposition~\ref{prop:syzygy-transfer} and Corollary~\ref{cor:two-sided}
with $s=0$. The same argument works with $J$ in place of $I$.
\end{proof}

Recall that $A$ is \emph{minimal representation-infinite} if it is
representation-infinite and every quotient by a nonzero ideal is
representation-finite.

\begin{corollary}\label{cor:minimal-repinfinite}
Let $A$ be minimal representation-infinite. Every interval for which
$P_{T,U}$ is nonfaithful satisfies all conclusions of
Corollary~\ref{cor:annihilator}. In particular, this holds for every
interval whose upper endpoint $U$ is nonfaithful.
\end{corollary}

\begin{proof}
Nonfaithfulness means $\operatorname{ann}_A P\neq0$.
For the last assertion, $P$ is a summand of a quotient of a finite sum
of copies of $U$, so $\operatorname{ann}_A U\subseteq
\operatorname{ann}_A P$.
\end{proof}

\begin{remark}
For $T=0$, the opposite-finitistic-dimension conclusion is related to
\cite[Corollary~5.7]{XuZhang}. The interval criterion is formulated
instead in terms of the reduced generator $P_{T,U}$, which need not be
support $\tau$-tilting as an $A$-module. The nonfaithfulness condition
must be retained; no assertion for faithful $P_{T,U}$ follows from this
quotient argument.
\end{remark}

\subsection{A hereditary bound}

\begin{proposition}[Second extensions]\label{prop:hereditary-bound}
Suppose that $\Ext_A^2(X,Y)=0$ for all $X,Y\in\HH_{T,U}$.
Then every first relative syzygy is relatively projective. Consequently,
\[
 \operatorname{glo}\text{-}k\text{-}\dell[T,U]\leq1
 \quad(k\geq1),\qquad
 \operatorname{gl.dim}B\leq2.
\]
In particular, this applies to every interval over any finite-dimensional
hereditary algebra, with no restriction on its representation type.
\end{proposition}

\begin{proof}
For $X\in\HH_{T,U}$, take a conflation
$0\to K\to P_0\to X\to0$ with $P_0\in\add P$.
For every $Y\in\HH_{T,U}$ the exact sequence
\[
 \Ext_A^1(P_0,Y)\longrightarrow\Ext_A^1(K,Y)
                  \longrightarrow\Ext_A^2(X,Y)
\]
has zero outer terms. Hence $K$ is Ext-projective in the heart,
so $K\in\add P$. This proves the assertion about relative syzygies
and their delooping levels. Now every first syzygy over $B$ is
$F(X)$ for some $X$ in the heart, while $F(X)$ has projective dimension
at most one by Proposition~\ref{prop:Yoneda}. Thus every $B$-module
has projective dimension at most two.
\end{proof}

The hereditary endomorphism-algebra bound is also consistent with the
classical theory of endomorphism algebras of rigid modules. The proof
above explains it directly in terms of interval syzygies and also gives
the local $\Ext^2$ criterion.

\begin{example}[Sharpness with a nonzero lower endpoint]
\label{ex:hereditary-sharp}
Let $A=\kk(1\to2\to3\to4)$, and denote its projectives and simples
by $P_i$ and $S_i$.  Set
\[
 T=S_4,\qquad U=P_1\oplus S_1\oplus P_3\oplus P_4.
\]
The only nonprojective summand of $U$ is $S_1$. Its projective
resolution $0\to P_2\to P_1\to S_1\to0$ shows
$\Ext_A^1(S_1,U)=0$: the only nonzero target contribution from
$\Hom_A(P_2,U)$ is the map into $P_1$, which extends to $P_1$.
Thus $U$ is a tilting module with four indecomposable summands, and
$T\leq U$. Taking traces of $S_4$ gives
\[
 P_{T,U}=M_{123}\oplus S_1\oplus S_3,
\]
where $M_{123}$ is the interval representation on vertices $1,2,3$.
The only nonzero Hom spaces between distinct summands are
\[
 S_3\longrightarrow M_{123}\longrightarrow S_1,
\]
both one-dimensional, and their composite is zero. Therefore, after
relabeling vertices,
\[
 B\cong\kk(1\longrightarrow2\longrightarrow3) /
           (\text{the path of length two}),
 \qquad \operatorname{gl.dim}B=2.
\]
This proves that the constant two cannot be improved even for a
connected ambient algebra and a nonzero lower endpoint.
\end{example}

\subsection{Frobenius hearts and an infinite family}

An exact category is \emph{Frobenius} if it has enough projectives and
injectives and these two classes coincide.

A finitely generated $B$-module is \emph{Gorenstein-projective} if
it is a cycle in an exact complex of finitely generated projectives
which remains exact after applying $\Hom_B(-,B)$.
Write $\operatorname{Gproj}B$ for their full subcategory, with the
inherited exact structure.  Recall that an algebra is
\emph{Iwanaga--Gorenstein} when its left and right self-injective
dimensions are finite.

\begin{proposition}[Frobenius hearts and Gorenstein structure]
\label{prop:frobenius}
If the inherited exact category $\HH_{T,U}$ is Frobenius, then $F$
induces an exact equivalence
\[
 \HH_{T,U}\xrightarrow{\sim}\operatorname{Gproj}B=\Sub B,
\]
and $B$ is Iwanaga--Gorenstein with
\[
 \operatorname{id}(B_B)\leq1,\qquad
 \operatorname{id}({}_BB)\leq1.
\]
Moreover, for every $k\geq1$,
\[
 \operatorname{glo}\text{-}k\text{-}\dell[T,U]=0,
 \qquad \operatorname{glo}\text{-}\dell B\leq1,
 \qquad \max\{\findim B,\findim B^{\op}\}\leq1.
\]
If the heart is also wide, $B$ is self-injective and all its objectwise
and global $k$-delooping levels, as well as both finitistic dimensions,
are zero.  The same conclusions hold for the hearts of
Theorem~\ref{thm:arbitrary-lower}.
\end{proposition}

\begin{proof}
For each $X$ in the heart, splice projective resolutions and injective
coresolutions to obtain a bi-infinite acyclic complex with terms in
$\add P$ and with $X$ as a cycle.  All its cycles lie in the heart.
Applying $F$ gives an exact complex of projective $B$-modules.
Full faithfulness identifies its $B$-dual with the complex obtained
by applying $\Hom_A(-,P)$, which is exact because $P$ is injective in
the heart.  Hence $F(X)$ is Gorenstein-projective.  This is the usual
Frobenius embedding construction; see \cite{ChenFrobenius}.

Every Gorenstein-projective module embeds into a projective module,
so it belongs to $F(\HH_{T,U})$ by submodule closure.
Conversely, Frobenius inflations embed every object of the image into
a projective.  Since every submodule of a projective already belongs
to the image, we obtain
\[
 F(\HH_{T,U})=\operatorname{Gproj}B=\Sub B.
\]
Proposition~\ref{prop:Yoneda}(5) gives the asserted exact equivalence.

In particular, every first $B$-syzygy is Gorenstein-projective.
Thus $\Ext_B^2(M,B)=0$ for every $M\in\modB$, and
$\operatorname{id}(B_B)\leq1$.
To obtain the left bound, let $I$ be a finitely generated injective
right $B$-module.  Choose exact sequences
\[
 0\longrightarrow G\longrightarrow Q\longrightarrow I\longrightarrow0,
 \qquad
 0\longrightarrow G\longrightarrow Q'\longrightarrow G'\longrightarrow0
\]
with $Q,Q'$ projective and $G,G'$ Gorenstein-projective.
Their pushout gives exact sequences
\[
 0\longrightarrow Q\longrightarrow E\longrightarrow G'\longrightarrow0,
 \qquad
 0\longrightarrow Q'\longrightarrow E\longrightarrow I\longrightarrow0.
\]
The first splits because $\Ext_B^1(G',Q)=0$, so $E$ is
Gorenstein-projective.  Embed $E$ into a projective $Q''$ with
Gorenstein-projective cokernel $G''$.  Then
\[
 0\longrightarrow I\longrightarrow Q''/Q'
       \longrightarrow G''\longrightarrow0
\]
splits, since $I$ is injective.  The module $Q''/Q'$ has projective
dimension at most one, hence so does $I$.  Taking $I=DB$ yields
$\operatorname{id}({}_BB)=\operatorname{pd}_B DB\leq1$.

Iterated inflations into projective-injectives show that every object
of the heart is stably an arbitrary higher relative syzygy.  Hence all
its global relative levels are zero.  Theorem~\ref{thm:global-sharp}
and Corollary~\ref{cor:global-comparison} give the bounds for the
global level of $B$ and for $\findim B^{\op}$.
If a right $B$-module has finite projective dimension, its first syzygy
is Gorenstein-projective of finite projective dimension, hence
projective.  This proves $\findim B\leq1$.
Finally, if the heart is wide, Theorem~\ref{thm:wide-equality}
identifies the Frobenius heart with $\modB$.  Thus $B$ is
self-injective, proving all the remaining vanishings.
\end{proof}

\begin{example}[A monomorphism heart]\label{ex:monomorphism}
Let $D$ be a basic self-injective algebra of finite dimension, with
$D\neq0$, and put
\[
 A=\kk(1\to2)\otimes_{\kk}D.
\]
Right $A$-modules are morphisms $f:X\to Y$ of right $D$-modules.
The regular module is
$A=(D\xrightarrow{1}D)\oplus(0\to D)$.
Take $U=A$ and $T=(D\to0)$.  The latter is the regular module over
the idempotent quotient supported at vertex $1$, so is support
$\tau$-tilting.  Since
\[
 \Hom_A(T,(X\xrightarrow{f}Y))\cong\Hom_D(D,\ker f),
\]
the heart is exactly the monomorphism category
\[
 \HH_{T,A}=\operatorname{Mon}(\operatorname{mod}D)
       =\{X\xrightarrow{f}Y\mid f\text{ is injective}\}.
\]
Also $\Hom_A(T,A)=0$, so $P_{T,A}=A$ and $B_{T,A}\cong A$.
This heart is not ICE-closed: the inclusion
$(0\to D)\hookrightarrow(D\xrightarrow{1}D)$ has cokernel $(D\to0)$.

The monomorphism category is Frobenius with projective-injectives
$\add A$; this is a special case of \cite[Lemma~2.1]{ChenMono}.
For completeness, its injective structure can be seen explicitly.
For a monomorphism $f:X\to Y$, embed $X$ in a projective-injective
$D$-module $I$ and $Y/fX$ in a projective-injective $D$-module $J$.
Extend the first embedding to $g:Y\to I$.  If $\pi:Y\to Y/fX$ is
the quotient and $j:Y/fX\hookrightarrow J$, the maps
\[
 X\hookrightarrow I,
 \qquad Y\longrightarrow I\oplus J,\quad y\longmapsto(g(y),j\pi(y))
\]
give an inflation of $f$ into $(I\hookrightarrow I\oplus J)$ with
monomorphic cokernel.  The objects $(I\xrightarrow{1}I)$ and
$(0\to J)$ are injective in the inherited exact structure:
Hom into them is respectively $\Hom_D(Y,I)$ and
$\Hom_D(\operatorname{coker}f,J)$, both exact on conflations.
They are also projective; together with the Ext-progenerator $A$ this
proves the Frobenius assertion.

Proposition~\ref{prop:frobenius} identifies this monomorphism heart
with $\operatorname{Gproj}A=\Sub A$ and shows that both self-injective
dimensions of $A$ are at most one. It also gives
\[
 \operatorname{glo}\text{-}k\text{-}\dell[T,A]=0,
 \qquad k\text{-}\dell A\leq1\quad(k\geq1).
\]
Moreover, $(D\to0)$ has projective dimension exactly one, as shown by
\[
 0\longrightarrow(0\to D)\longrightarrow(D\xrightarrow{1}D)
   \longrightarrow(D\to0)\longrightarrow0.
\]
Applying the same argument to $D^{\op}$, and reversing the vertices
to identify the opposite triangular algebra, proves
\[
 \findim A=\findim A^{\op}=1.
\]
Since $\findim A^{\op}\leq\dell A\leq k\text{-}\dell A\leq1$,
one also has $k\text{-}\dell A=1$ for every $k\geq1$.
Hence the interval bound is sharp for every such $D$.
These triangular-algebra homological facts are classical; the point here
is their realization by a nonwide interval whose intrinsic global
levels vanish.

Take, more concretely, an algebraically closed field of characteristic
not two and
\[
 D=\kk\langle x,y\rangle/(x^2,y^2,xy+yx).
\]
This exterior algebra is self-injective: the coefficient of $xy$
defines a nondegenerate associative bilinear form.  It is
representation-infinite, since the two-dimensional modules on which
$x$ acts as a fixed nonzero nilpotent matrix $N$ with $N^2=0$ and
$y$ as $\lambda N$ are pairwise nonisomorphic indecomposables as
$\lambda$ ranges over $\kk$.  Thus $A$ is connected and
representation-infinite.  It has infinite global dimension: a
nonprojective simple $D$-module has infinite projective dimension, and
evaluation at the second vertex takes projective $A$-modules to
projective $D$-modules.

The heart is representation-infinite because $X\mapsto(0\to X)$
preserves indecomposables and their isomorphism classes. Since it is
Frobenius, every object is, up to projective summands, an $r$th relative
syzygy for every $r$. Thus $\CC_r=\HH_{T,A}$ for all $r\geq0$,
and none of these strata is representation-finite. This verifies that
the Frobenius application goes beyond the finite-stratum criterion.
\end{example}

\section{Examples and further questions}
\label{sec:examples}

The examples in this section separate the phenomena occurring in the
comparison theorems. We begin with wide intervals of levels zero and
one, then exhibit the radical correction term and a connected non-ICE
example attaining the global bound. A product construction realizes
arbitrary classical levels. Finally, a Kronecker interval exhibits
a lower endpoint that is not functorially finite.

\subsection{A semisimple wide interval}

\begin{example}\label{ex:wide-A2}
Let $A=\kk(1\to2)$, with simples $S_i$ and projectives $P_i$.
Thus $P_2=S_2$ and $\rad P_1=S_2$.  Take $T=S_2$ and $U=A$.
Since $\Hom_A(S_2,X)\cong X_2$, the heart is
\[
 \HH_{T,A}=S_2^\perp=\add S_1.
\]
The traces in $P_1$ and $P_2$ are respectively $\rad P_1$ and $P_2$,
so
\[
 P_{T,A}=A/\tr_{S_2}(A)\cong S_1,
 \qquad B_{T,A}\cong\kk.
\]
This is a semisimple wide interval with nonzero lower endpoint.
All objectwise, finite-test, and global relative $k$-delooping levels
vanish, illustrating Theorem~\ref{thm:wide-equality}.
\end{example}

\subsection{A wide interval of positive level}

The preceding heart is semisimple. The next example shows that a
nonzero lower endpoint can also produce a nonsemisimple heart and a
positive interval level without any radical correction term.

\begin{example}\label{ex:wide-positive}
Let $A=\kk Q$ for
\[
 Q:1\longrightarrow2\longrightarrow3,
\]
and write $S_i$ and $P_i$ for the simple and indecomposable projective
right $A$-modules at vertex $i$.  Take
\[
 T=S_3,\qquad U=A=P_1\oplus P_2\oplus P_3.
\]
The module $T$ is support $\tau$-tilting and $\Fac T=\add S_3$.  Since
$\Hom_A(S_3,X)\cong X_3$ for a representation $X$ of $Q$, one has
\[
 \HH_{T,U}=S_3^{\perp}
 =\{X\in\modA\mid X_3=0\}
 \simeq\operatorname{mod}\kk(1\longrightarrow2).
\]
In particular, the interval is wide. The traces of $S_3$ in the three
indecomposable projectives give
\[
 P_{T,U}=A/\tr_{S_3}(A)
 \cong (P_1/S_3)\oplus(P_2/S_3),
\]
which is the regular projective generator of the displayed $A_2$ heart.
Consequently,
\[
 B_{T,U}=\End_A(P_{T,U})\cong\kk(1\longrightarrow2).
\]

Inside the heart there is a nonsplit exact sequence
\[
 0\longrightarrow S_2\longrightarrow P_1/S_3
 \longrightarrow S_1\longrightarrow0.
\]
Thus $S_2$ and $P_1/S_3$ are projective in the heart and
$\Omega_{T,U}S_1\cong S_2$. Since the heart is hereditary, every first
syzygy is projective. It follows that, for every $k\geq1$,
\[
 k\text{-}\dell_{[T,U]}S_1=1,
 \qquad
 \operatorname{glo}\text{-}k\text{-}\dell[T,U]=1.
\]
Indeed, the value is at most one because the first syzygy of $S_1$ is
projective, and it is not zero because a nonprojective object cannot be
a stable retract of a projective object.  Theorem~\ref{thm:wide-equality}
therefore gives
\[
 k\text{-}\dell^{\mathrm{XZ}}[T,U]
 =k\text{-}\dell B_{T,U}=1,
 \qquad
 \findim B_{T,U}^{\op}=1.
\]
This is a non-product example in which all wide-interval equalities are
nontrivial.
\end{example}

\subsection{The radical correction term is necessary}

\begin{example}\label{ex:correction}
Keep $A=\kk(1\to2)$ and put
\[
  T=0,\qquad U=P_1\oplus S_1.
\]
The module $U$ is a basic tilting, hence $\tau$-tilting, module.  Its
torsion class is
\[
  \HH_{0,U}=\Fac U=\add(P_1\oplus S_1).
\]
It is ICE-closed but not wide: the kernel of the canonical epimorphism
$P_1\twoheadrightarrow S_1$ is $S_2$, which does not belong to the
heart. Since every object of the heart lies in $\add U$, its inherited
exact structure is split and every objectwise relative $k$-delooping
level is zero.

Put $P=U$, $B=\End_A(U)$ and $F=\Hom_A(U,-)$.  The algebra $B$ is a
hereditary algebra of Dynkin type $A_2$.  With
\[
 Q_1=F(P_1),\qquad Q_2=F(S_1),
\]
the module $Q_1$ is simple projective and
\[
  \rad Q_2\cong Q_1.
\]
Thus $L_1=Q_1$ belongs to $F(\HH_{0,U})$, while
$L_2=\topm Q_2$ does not. The correction object for $L_2$ is
$N_2=P_1$, because $F(P_1)=Q_1\cong\rad Q_2$.  Consequently,
\[
 \dell^{\mathrm{XZ}}[0,U]
 =\max\{\dell_{[0,U]}P_1,
         2\text{-}\dell_{[0,U]}P_1+1\}
 =1.
\]
On the other hand, $\dell B=1$. Hence the comparison theorem is sharp,
and omitting the radical correction term would give the false bound
$\dell B\leq0$.
\end{example}

\subsection{A connected non-ICE interval attaining the bound}

\begin{example}\label{ex:nonice}
Let $A=\kk(1\to2)$, but now take $T=S_1$ and $U=A$.
The module $S_1$ is support $\tau$-tilting, since it is the regular
module of the quotient supported at vertex $1$. Among the three
indecomposable $A$-modules, precisely $P_1$ and $S_2$ receive no
nonzero map from $S_1$.  Therefore
\[
 \HH_{T,A}=S_1^\perp=\add(P_1\oplus S_2)=\proj A.
\]
This heart is not ICE-closed: the inclusion $S_2\hookrightarrow P_1$
has cokernel $S_1\notin\HH_{T,A}$.  Also
$\Hom_A(S_1,A)=0$, so $\tr_{S_1}(A)=0$, $P_{T,A}=A$ and
$B_{T,A}=\End_A(A)\cong A$. he functor $F$ is the usual identity
identification on $\proj A$; its image is submodule-closed because $A$
is hereditary.

Every object of the heart is relatively projective, so
$\operatorname{glo}\text{-}k\text{-}\dell[T,A]=0$ for every $k\geq1$.
The simple $S_2$ lies in the image, whereas $S_1$ does not, and
$\rad P_1=S_2$ provides its correction object. Consequently
\[
 1=\findim B_{T,A}^{\op}
  =\dell B_{T,A}
  =\dell^{\mathrm{XZ}}[T,A]
  =\operatorname{glo}\text{-}\dell[T,A]+1.
\]
Also $\operatorname{glo}\text{-}\dell B_{T,A}=1$, since $B_{T,A}$
is hereditary and has a nonprojective simple. Hence the upper bound
in Theorem~\ref{thm:global-sharp} is attained. The unrestricted
comparison theorem applies strictly beyond ICE intervals, and its
additive constant is necessary even when $A$ is connected and the
lower endpoint is nonzero. Example~\ref{ex:wide-positive} attains
the other possibility, namely equality of the two global levels.
\end{example}

\subsection{Product lifting and realization}

Example~\ref{ex:correction} starts at the zero torsion class. The
following elementary product construction turns every Xu--Zhang example
into an interval with nonzero lower endpoint as soon as the second factor
is nonzero.

\begin{proposition}[Product lifting]\label{prop:product}
Let $A_0$ and $C$ be finite-dimensional basic algebras, let
$V\in\stau A_0$, and put $A=A_0\times C$. Consider the support
$\tau$-tilting $A$-modules
\[
  T=(0,C),\qquad U=(V,C).
\]
Then
\[
 \HH_{T,U}=\Fac_{A_0}V\times0,
 \qquad
 P_{T,U}=(V,0).
\]
Moreover,
\[
 \End_A(P_{T,U})\cong\End_{A_0}(V),
\]
and all objectwise, global, and Xu--Zhang interval delooping levels agree
with the corresponding relative levels for $\Fac_{A_0}V$ with respect
to $V$.
\end{proposition}

\begin{proof}
We have
\[
 \Fac_A T=0\times\mod C,
 \qquad
 \Fac_A U=\Fac_{A_0}V\times\mod C,
\]
and
\[
 T^{\perp}=\mod A_0\times0.
\]
This proves the formula for the heart. The trace of $(0,C)$ in $(V,C)$
is $(0,C)$, so the formula for $P_{T,U}$ follows. The remaining
statements are immediate from the componentwise exact structure and
syzygy functor.
\end{proof}

The product construction also shows that interval levels with nonzero
lower endpoint contain the entire classical theory.

\begin{corollary}[Realization of classical levels]\label{cor:realization}
Let $\Lambda$ and $C$ be finite-dimensional basic algebras with $C\neq0$, and
put
\[
 A=\Lambda\times C,\qquad T=(0,C),\qquad U=(\Lambda,C)=A.
\]
Then $[\Fac T,\Fac U]$ is a wide interval with nonzero lower endpoint,
\[
 \HH_{T,U}=\operatorname{mod}\Lambda\times0,
 \qquad
 P_{T,U}=(\Lambda,0),
 \qquad
 B_{T,U}\cong\Lambda.
\]
For every $k\geq1$,
\[
 \operatorname{glo}\text{-}k\text{-}\dell[T,U]
 =\operatorname{glo}\text{-}k\text{-}\dell\Lambda,
 \qquad
 k\text{-}\dell^{\mathrm{XZ}}[T,U]
 =k\text{-}\dell\Lambda.
\]
\end{corollary}

\begin{proof}
Apply Proposition~\ref{prop:product} with $A_0=\Lambda$ and
$V=\Lambda$. The heart is the module category of $\Lambda$, hence is
wide, and Theorem~\ref{thm:wide-equality} gives the asserted equalities.
\end{proof}

For example, taking $\Lambda=\kk(1\to2)$ realizes the value one, while
taking $\Lambda$ to be the basic algebra constructed by Kershaw--Rickard
\cite{KershawRickard} gives a wide interval with nonzero lower endpoint
and infinite delooping level. Thus finiteness of interval levels is not
automatic even for wide intervals.

\subsection{A lower endpoint that is not functorially finite}

\begin{example}[A Kronecker interval]\label{ex:kronecker-lower}
Let $A$ be the Kronecker algebra over an algebraically closed field,
and let $\mathcal P$, $\mathcal R$, and $\mathcal I$ denote the
additive subcategories of preprojective, regular, and preinjective
modules. The classical separation of these components gives the
split torsion pair
\[
 (\mathcal V,\mathcal P),\qquad
 \mathcal V=\add(\mathcal R\cup\mathcal I);
\]
see, for example, \cite[Section~2]{AngeleriHuberyKrause}.
With upper endpoint $U=A$, the interval heart is
\[
 \HH_{\mathcal V,A}=\mathcal V^\perp=\mathcal P.
\]
All projective modules are preprojective, so $t_{\mathcal V}(A)=0$.
Consequently, $P=A$, $B=A$, and the realization of
Theorem~\ref{thm:arbitrary-lower} is the inclusion
$\mathcal P\hookrightarrow\modA$.

The lower torsion class $\mathcal V$ is not functorially finite.
Indeed, otherwise $\mathcal V=\Fac M$ for a finite-dimensional
module $M\in\mathcal V$, by \cite[Theorem~2.7]{AIR}.
Only finitely many regular tubes contain summands of $M$.
Choose a regular simple $R_\lambda$ in another homogeneous tube.
There are no maps from the preinjective summands of $M$ to
$R_\lambda$, and distinct regular tubes are Hom-orthogonal.
Thus $\Hom_A(M,R_\lambda)=0$, contradicting
$R_\lambda\in\Fac M$.

Since $A$ is hereditary, every first relative syzygy is projective.
The heart has nonprojective preprojective modules, so
\[
 \CC_1(\HH_{\mathcal V,A})=\proj A,
 \qquad
 \operatorname{glo}\text{-}\dell\HH_{\mathcal V,A}
 =\operatorname{glo}\text{-}\dell A=1.
\]
The heart itself has infinitely many indecomposables and is not
functorially finite: otherwise \cite[Lemma~4.18]{EnomotoSakai}
would make $\mathcal V$ functorially finite as well. The same lemma
shows that a heart of an interval with two functorially finite
endpoints is functorially finite. Thus this particular heart cannot
be obtained from two support $\tau$-tilting endpoints. The extension
to arbitrary lower torsion classes therefore enlarges the class of
hearts to which the comparison applies.
\end{example}

\subsection{Further questions}

We conclude with two questions and their relation to recent work.

\begin{question}\label{q:equality}
Characterize the intervals for which
\[
 k\text{-}\dell_BF(X)=k\text{-}\dell_{[T,U]}X
\]
for every object
$X$ of the heart and every $k\geq1$.
\end{question}

Wide intervals have this property. Proposition~\ref{prop:sandwich}
proves equality whenever the classical $k$- and $(k+1)$-delooping
levels of $F(X)$ agree.  Equality does not hold for arbitrary
intervals: \cite[Corollary~4.10]{XuZhangHigher} gives arbitrarily
large gaps already for a zero lower endpoint, and
Proposition~\ref{prop:product} carries them to nonzero lower endpoints.
The question is therefore to find structural equality criteria for
nonwide hearts, extending the sufficient conditions for $\Fac U$ in
\cite[Proposition~3.4]{XuZhangHigher}.

\begin{question}\label{q:monotonicity}
Under what hypotheses does inclusion of intervals control their global
relative delooping levels?  More precisely, when does
\[
 [\Fac T_1,\Fac U_1]\subseteq[\Fac T_2,\Fac U_2]
\]
imply an inequality between
$\operatorname{glo}\text{-}\dell[T_1,U_1]$ and
$\operatorname{glo}\text{-}\dell[T_2,U_2]$?
\end{question}

There is no formal monotonicity because both the heart and its
Ext-progenerator change. For a $\tau$-tilting finite algebra one can
define a monotone envelope by setting
\[
 \mathfrak D[T,U]
 =\sup_{T\leq V\leq W\leq U}
   \operatorname{glo}\text{-}\dell[V,W].
\]
This envelope is monotone under interval inclusion by construction.
Whether it distinguishes intervals with the same exact heart requires
further examples; no such assertion is needed here.

It would also be useful to compare interval invariants under
$\tau$-tilting reduction in the sense of Jasso \cite{Jasso}.
For the wide intervals arising in that reduction,
Theorem~\ref{thm:wide-equality} identifies the relative invariants with
the classical invariants of the corresponding module category.
The challenge is to obtain comparable statements for nonwide hearts.

\medskip
\noindent\textbf{Disclosure of computational assistance.}\quad
ChatGPT was used during revision to assist with 
literature searches, language editing, and manuscript preparation.

\medskip
\noindent\textbf{Acknowledgements.}\quad
Weiqing Cao is supported by the
National Natural Science Foundation of China (Grant No.~12601071). Mingfei Xu and Xiaojin Zhang are supported by the National Natural Science Foundation of
China (Grant Nos.~12171207 and 12371038).

\medskip
\noindent\textbf{Data availability.}\quad
Data sharing is not applicable to this article, as no datasets were
generated or analysed during the current study.

\medskip
\noindent\textbf{Conflict of interest.}\quad
The authors declare that they have no conflicts of interest.

\end{document}